\documentclass{article}
\usepackage[final]{graphicx}
\usepackage{psfrag}
\usepackage{amsmath,amsfonts,amssymb,amsxtra,subeqnarray,bm}

\newlength{\extramargin}
\newcommand{\Real}{\ensuremath{{\mathbb{R}}}}

\newcommand{\Complex}{\ensuremath{{\mathbb{C}}}}

\newcommand{\C}{\ensuremath{\mathcal C}}

\newcommand{\setP}{\ensuremath{\mathcal P}}

\newcommand{\A}{\ensuremath{\mathcal A}}

\newcommand{\V}{\ensuremath{\mathcal V}}

\newcommand{\setS}{\ensuremath{\mathcal S}}

\newcommand{\setE}{\ensuremath{\mathcal E}}

\newcommand{\G}{\ensuremath{\mathcal G}}

\newcommand{\F}{\ensuremath{\mathcal F}}
\newcommand{\T}{\ensuremath{\mathcal T}}
\renewcommand{\L}{\ensuremath{\mathcal L}}

\newcommand{\one}{\ensuremath{{\mathbf{1}}}}

\newtheorem{theorem}{Theorem}

\newtheorem{lemma}{Lemma}

\newtheorem{definition}{Definition}
\newtheorem{remark}{Remark}
\newtheorem{assumption}{Assumption}

\newenvironment{proof}{\noindent {\bf Proof.}}{\hfill \hspace*{1pt}\hfill$\blacksquare$}

\begin{document}
\title{Synchronization of oscillators not sharing a common ground}
\author{S. Emre Tuna\footnote{The author is with Department of
Electrical and Electronics Engineering, Middle East Technical
University, 06800 Ankara, Turkey. Email: {\tt etuna@metu.edu.tr}}}
\maketitle

\begin{abstract}
Networks of coupled LC oscillators that do not share a common ground
node are studied. Both resistive coupling and inductive coupling are
considered. For networks under resistive coupling, it is shown that
the oscillator-coupler interconnection has to be bilayer if the
oscillator voltages are to asymptotically synchronize. Also, for
bilayer architecture (when both resistive and inductive couplers are
present) a method is proposed to compute a complex-valued effective
Laplacian matrix that represents the overall coupling. It is proved
that the oscillators display synchronous behavior if and only if the
effective Laplacian has a single eigenvalue on the imaginary axis.
\end{abstract}

\section{Introduction}

A network of two-terminal electrical oscillators coupled via
two-terminal components gives rise to a pair of graphs. One of them
is the coupler graph, whose edges represent the couplers. The other
is the oscillator graph, where the edges stand for the oscillators.
As an illustration, let us reproduce in Fig.~\ref{fig:chua}a the
example array of six Chua's oscillators coupled via linear resistors
presented in \cite{wu01}. The corresponding coupler graph and
oscillator graph are given in Fig.~\ref{fig:chua}b and
Fig.~\ref{fig:chua}c, respectively. Note that, since one terminal of
each oscillator rests on the so called ground node
$\textcircled{g}$, the oscillator graph in this example is a
star.\footnote{I.e., a tree with diameter no larger than two.} Such
networks with a star oscillator graph have been considered, for
instance, in \cite{vandersteen06,dorfler14,narahara19,cejnar20}.
Another type of topology appears in the works
\cite{peles03,achanta18,liu20} which study synchronization in a
series-connected array of oscillators. Those networks enjoy an
oscillator graph that is a path; see Fig.~\ref{fig:vdp}.

\begin{figure}[h]
\begin{center}
\includegraphics[scale=0.5]{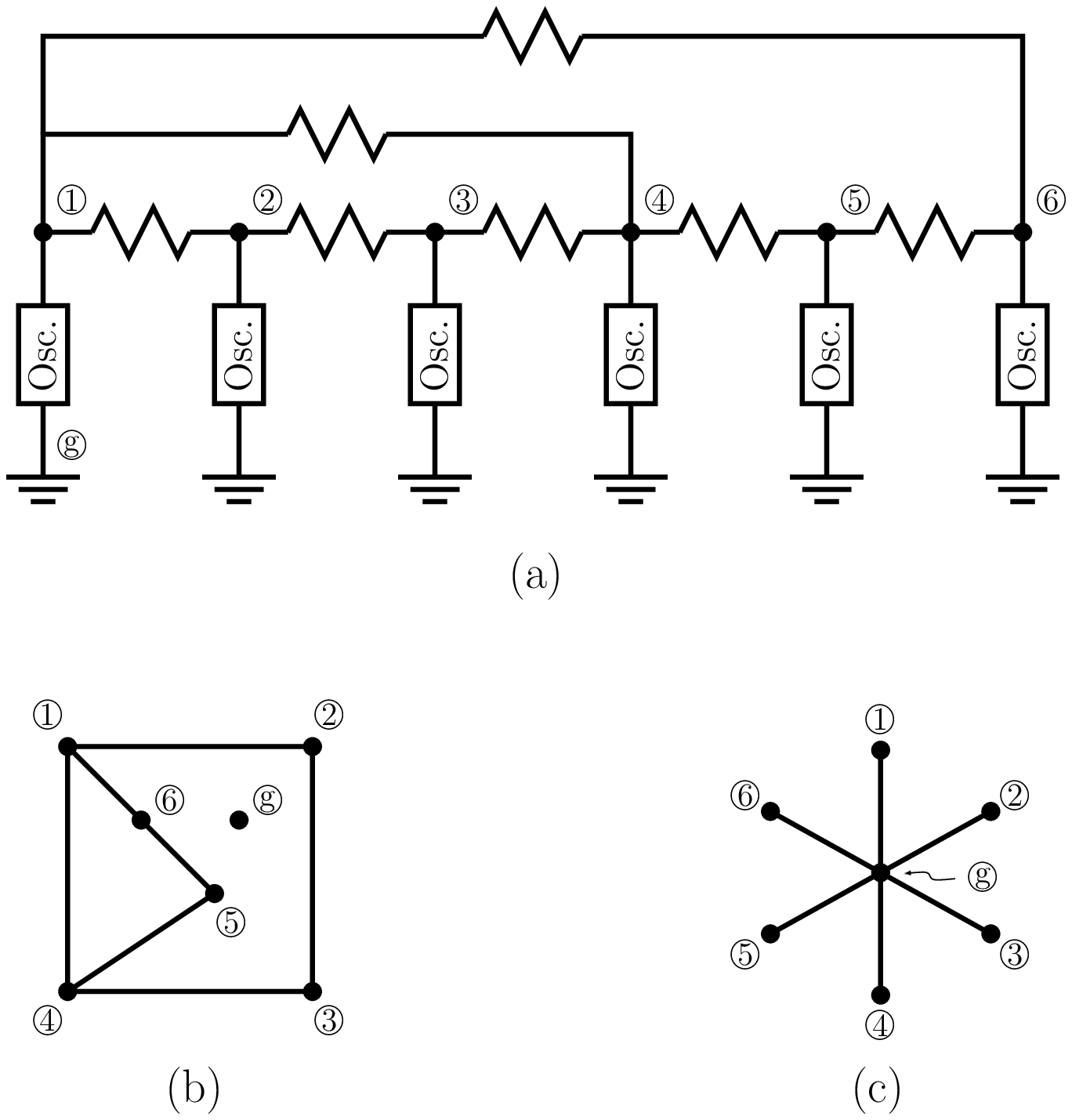}
\caption{(a) An array of coupled oscillators. (b) The coupler graph.
(c) The oscillator graph.}\label{fig:chua}
\end{center}
\end{figure}

The majority of research effort on electrical networks focuses on
those with a star oscillator graph. There is a good reason for that.
Man-made systems, be it a large power network or a tiny microchip,
are usually designed so that there is a common ground node in the
overall circuit to which the units (generators, oscillators, etc.)
are directly connected. Still, this real-world significance of
networks with a common ground, we believe, is no justification for
the general neglect in the literature of less restrictive
architectures. Let us briefly speculate why. Clearly, not having to
be confined to star topology means more flexibility both in design
and in analysis. Flexibility in design can be important, for
instance, if the system to be built is part of a microchip, where
the design constraints are already very tight. And, flexibility in
analysis may prove useful, e.g., in understanding natural phenomena:
Certain biological systems are long known to be able to be modeled
by interconnected electrical oscillators and it is very unlikely
that nature should have a certain preference for networks where the
oscillators share a common ground node. In addition to these
practical benefits, understanding the general case has value in its
own right. Motivated by these, we study in this paper the collective
behavior of coupled LC circuits without assuming that the oscillator
graph is a star; two simple examples are shown in
Fig.~\ref{fig:cross}.

\begin{figure}[h]
\begin{center}
\includegraphics[scale=0.5]{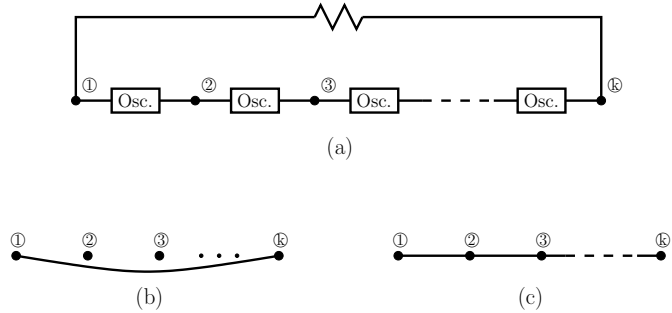}
\caption{(a) An array of series-connected oscillators. (b) The
coupler graph. (c) The oscillator graph.}\label{fig:vdp}
\end{center}
\end{figure}

\begin{figure}[h]
\begin{center}
\includegraphics[scale=0.5]{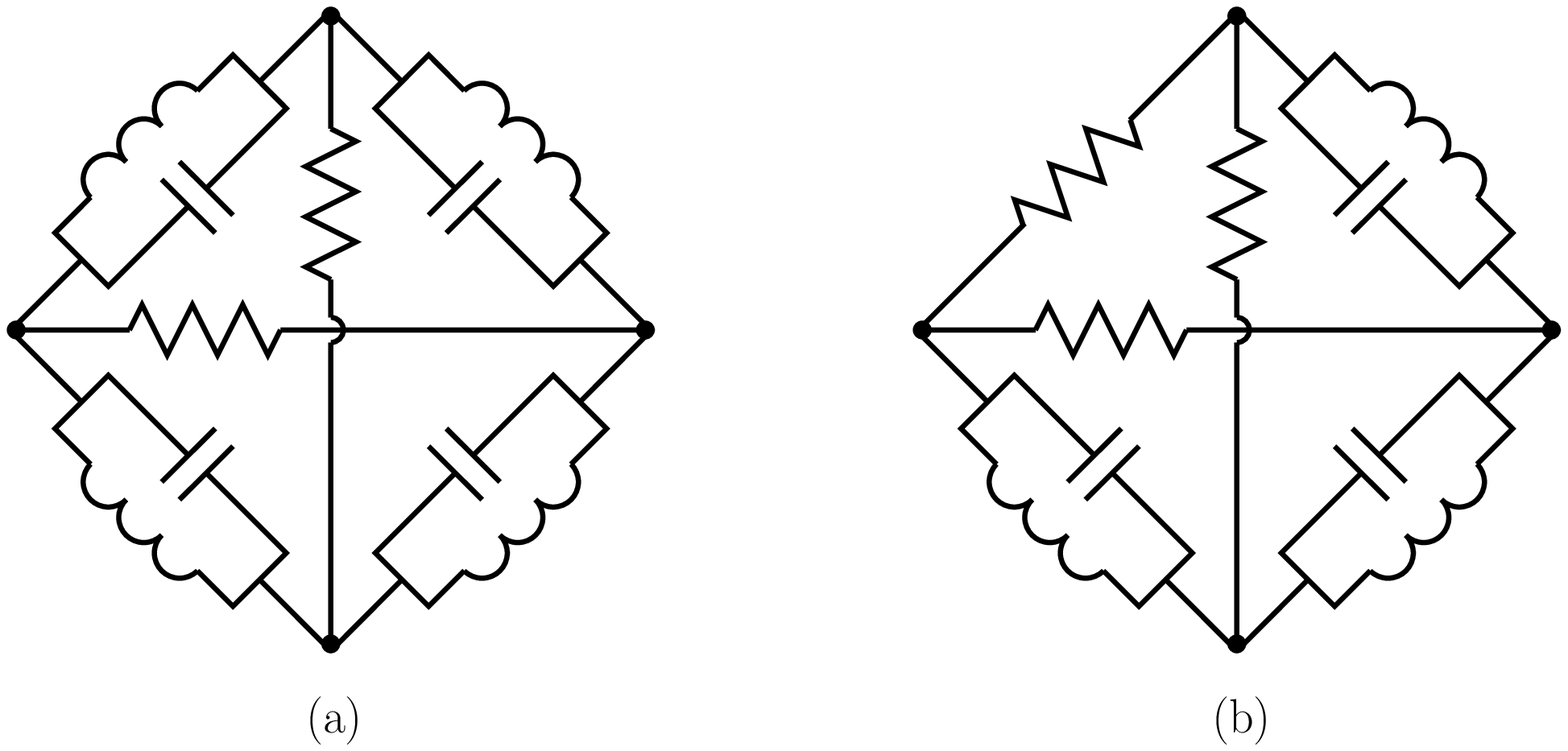}
\caption{Networks of coupled LC tanks with nonstar oscillator
graph.}\label{fig:cross}
\end{center}
\end{figure}

We begin our investigation by studying the linear time invariant
(LTI) network of identical LC oscillators coupled by resistors only.
We show that such a network displays synchronous behavior if and
only if the oscillators communicate through a bilayer coupling
structure (see Fig.~\ref{fig:bilayer}) and the graphs representing
the layers are both connected. In the second half of the paper we
focus our attention on a more general situation, where not only
resistors but also inductors are allowed as couplers. In this case
the overall coupling gives rise to four Laplacian matrices
$G_{1},\,G_{2},\,B_{1},\,B_{2}$. The matrices $G_{1}$ and $G_{2}$
represent the resistors in the first and second layers,
respectively; $B_{1}$ and $B_{2}$ represent the inductors in the
first and second layers, respectively. We propose a method to
generate a single complex matrix (which we call the effective
Laplacian) out of these four real matrices\footnote{The situation is
a bit subtler. There are indeed six (not four) matrices to be taken
into consideration. The details are given in
Section~\ref{sec:reig}.} and study some of its properties. Then we
show that the oscillators asymptotically synchronize if and only if
this effective Laplacian has a single eigenvalue on the imaginary
axis.

\begin{figure}[h]
\begin{center}
\includegraphics[scale=0.5]{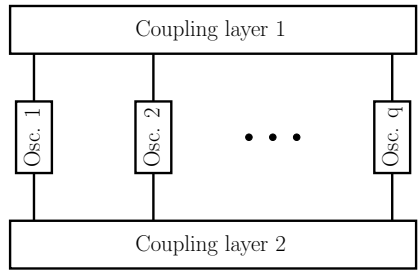}
\caption{A network of oscillators with bilayer coupling
structure.}\label{fig:bilayer}
\end{center}
\end{figure}

Possible contributions of this paper are intended to be in two
places. First. Through this preliminary work, we hope to draw
attention to nonstar networks and the possible riches they may
contain. At first, their apparent lack of structure might suggest
that nothing interesting will come out of them. To this view,
however, we present a counterevidence in this paper. As mentioned
earlier, when one attempts to extract synchronous behavior from a
nonstar network, an elegant structure (bilayer coupling) emerges as
a necessary condition (Theorem~\ref{thm:bl}). Possibly, other
different phenomena flourish on other interesting structures; and
considering networks where oscillators do not share a common ground
node would allow one to discover those interconnection forms. This
is somewhat in contrast to the general style in the literature,
where the analysis usually rests on an initially assumed
architecture, e.g., networks with star oscillator graph. Second. We
make a first step toward a systematic approach (which involves
studying the spectral properties of the effective Laplacian) for
understanding the joint tendencies of electrical oscillators under
bilayer coupling. To the best of our knowledge, this is a novelty,
for the mature literature on synchronization of harmonic oscillators
(see, e.g., \cite{ren08,su09,zhou12,tuna17}) does not seem to
provide one with off-the-shelf tools to determine the asymptotic
behavior of coupled LC tanks in the absence of a common ground node.

\section{Notation}

Let $I\in\Real^{n\times n}$ denote the identity matrix and
$I(:,\,r)\in\Real^{n}$ be its $r^{\rm th}$ column. The vector of all
ones is denoted by $\one_{q}\in\Real^{q}$. The set of eigenvalues of
a square matrix $P$ is denoted by ${\rm eig}(P)$ and its
pseudoinverse by $P^{+}$. A matrix $F\in\Real^{r\times\ell}$ is said
to be class~$\F$ if it satisfies the following properties: (i) each
entry of $F$ is either 0 or 1, (ii) each column of $F$ contains a
single nonzero entry, and (iii) $F$ has no zero rows. Note that
$F^{T}\one_{r}=\one_{\ell}$ when $F$ is class~$\F$. In this paper we
assume that the reader is familiar with some basic graph theoretic
terms such as {\em tree}, {\em path}, and {\em cycle}.

\section{Problem statement}

Recall that the voltage $v_{k}\in\Real$ of an uncoupled LC
oscillator satisfies $c{\ddot v}_{k}+l^{-1}v_{k}=0$, where $l,\,c>0$
are the associated inductance and capacitance, respectively.
Consider now a network of $q$ identical LC oscillators coupled by
LTI resistors and inductors. This network has $n$ nodes (denoted by
$\nu_{1},\,\nu_{2},\,\ldots,\,\nu_{n}$) and each node is assumed to
be incident to at least one oscillator. The node voltages (defined
with respect to some arbitrary reference point) are denoted by
$e_{1},\,e_{2},\,\ldots,\,e_{n}\in\Real$. Let $A\in\Real^{n\times
q}$ be the (oriented) incidence matrix associated to the
oscillators. This matrix is constructed as follows. Suppose the
$k^{\rm th}$ oscillator extends between the nodes $\nu_{r}$ and
$\nu_{s}$\footnote{We allow at most one oscillator extending between
any pair of nodes. In other words, no two oscillators can be
parallel.}; and the positive terminal of the oscillator rests on
$\nu_{r}$ while the negative terminal on $\nu_{s}$. (Note that this
implies: the voltage $v_{k}$ of the $k^{\rm th}$ oscillator reads
$v_{k}=e_{r}-e_{s}$.) Then the $k^{\rm th}$ column of $A$ satisfies
$A(:,\,k)=I(:,\,r)-I(:,\,s)$. Since each node is incident to at
least one oscillator, $A$ has no zero rows. We denote the $sk^{\rm
th}$ entry of $A$ by $\alpha_{sk}$. Let $g_{sr}=g_{rs}\geq 0$ be the
conductance of the resistor connecting the nodes $\nu_{r}$ and
$\nu_{s}$. ($g_{rs}=0$ means there is no resistor between $\nu_{r}$
and $\nu_{s}$.) Likewise, let $b_{sr}=b_{rs}\geq 0$ denote the
reciprocal ($b_{rs}=l_{rs}^{-1}$) of the inductance $l_{rs}$ of the
inductor connecting $\nu_{r}$ and $\nu_{s}$.\footnote{Recall that
the susceptance of an LTI inductor with inductance $l_{rs}$ at some
frequency $\omega$ equals $-(\omega l_{rs})^{-1}$. Hence the
parameter $b_{rs}$ is the magnitude of the susceptance of the
inductor at the frequency $\omega=1\,$rad/s.} ($b_{rs}=0$ means
there is no inductor between $\nu_{r}$ and $\nu_{s}$.) The set of
equations describing the evolution of this network then reads
\begin{subeqnarray}\label{eqn:motion}
\sum_{k=1}^{q}\alpha_{rk}(c{\ddot v}_{k}+l^{-1}
v_{k})+\sum_{s=1}^{n}g_{rs}({\dot e}_{r}-{\dot
e}_{s})+\sum_{s=1}^{n}b_{rs}(e_{r}-e_{s})=0\,,\qquad
r=1,\,2,\,\ldots,\,n\\
v_{k}=\sum_{r=1}^{n}\alpha_{rk}e_{r}\,,\qquad k=1,\,2,\,\ldots,\,q
\end{subeqnarray}

\begin{definition}\label{def:sync}
The network of oscillators~\eqref{eqn:motion} is said to be {\em
(nontrivially) synchronous} if $|v_{k}(t)|-|v_{\ell}(t)|\to 0$ as
$t\to\infty$ for all pairs $(k,\,\ell)$ and all initial conditions;
and $v_{k}(t)\not\to 0$ for some $k$ and some initial conditions.
\end{definition}

\begin{remark}\label{rem:polarity}
If the network~\eqref{eqn:motion} is synchronous, then the new
network obtained by reversing the polarities of some oscillator
voltages (i.e., letting $v_{k}^{\rm new}=-v_{k}$ for some $k$) is
still synchronous. In other words, whether a network is synchronous
or not is independent of how the oscillator voltage polarities are
chosen.
\end{remark}

A simple example for a synchronous network is given in
Fig.~\ref{fig:cross}a, where the oscillator voltages (nontrivially)
satisfy $|v_{k}(t)|-|v_{\ell}(t)|\to 0$. The network in
Fig.~\ref{fig:cross}b, however, is not synchronous because
$v_{k}(t)\to 0$ for all $k$.

To simplify analysis let us first construct the vectors $v=[v_{1}\
v_{2}\ \cdots\ v_{q}]^{T}$ and $e=[e_{1}\ e_{2}\ \cdots\
e_{n}]^{T}$. These two vectors are related to one another through
the identity $v=A^{T}e$. Secondly, let us introduce the $n\times n$
Laplacian matrices
\begin{eqnarray*}
G =
\left[\begin{array}{cccc}\sum_{s}g_{1s}&-g_{12}&\cdots&-g_{1n}\\
-g_{21}&\sum_{s}g_{2s}&\cdots&-g_{2n}\\
\vdots&\vdots&\ddots&\vdots\\
-g_{n1}&-g_{n2}&\cdots&\sum_{s}g_{ns}\end{array}\right]\,,\qquad B =
\left[\begin{array}{cccc}\sum_{s}b_{1s}&-b_{12}&\cdots&-b_{1n}\\
-b_{21}&\sum_{s}b_{2s}&\cdots&-b_{2n}\\
\vdots&\vdots&\ddots&\vdots\\
-b_{n1}&-b_{n2}&\cdots&\sum_{s}b_{ns}\end{array}\right]\,.
\end{eqnarray*}
Observe that $G$ and $B$ are symmetric and positive semidefinite.
Thirdly, we let (without loss of generality) $c=1$ and define
$\omega_{0}=1/\sqrt{lc}$. We can now rewrite \eqref{eqn:motion} as
\begin{subeqnarray}\label{eqn:nonstar}
AA^{T}({\ddot e}+\omega_{0}^{2}e)+G{\dot e}+Be&=&0\\
v&=&A^{T}e\,.
\end{subeqnarray}
The problem we intend to solve in this paper is this.

\begin{center}
{\em Find conditions on the triple $(A,\,G,\,B)$ under which the
network~\eqref{eqn:motion} is synchronous.}
\end{center}

Note that the network we study comprises LTI passive components.
Therefore it is obvious physically that the solutions have to be
bounded. The following remark formalizes this simple observation and
will prove useful for later analysis.

\begin{remark}\label{rem:passive}
Let $v(t)=A^{T}e(t)$ be an arbitrary solution of the
network~\eqref{eqn:nonstar}. Construct the nonnegative function
$W(t)=\frac{1}{2}e(t)^{T}Be(t)+\frac{1}{2}\omega_{0}^{2}v(t)^{T}v(t)+\frac{1}{2}{\dot
v}(t)^{T}{\dot
v}(t)=\frac{1}{2}e(t)^{T}Be(t)+\frac{1}{2}\omega_{0}^{2}e(t)^{T}AA^{T}e(t)+\frac{1}{2}{\dot
e}(t)^{T}AA^{T}{\dot e}(t)$. Computing the time derivative of $W$
along the solutions of \eqref{eqn:nonstar} we obtain
\begin{eqnarray}\label{eqn:Wdot}
{\dot W}
&=&{\dot e}^{T}Be+\omega_{0}^{2}{\dot e}^{T}AA^{T}e+{\dot e}^{T}AA^{T}{\ddot e}\nonumber\\
&=&{\dot e}^{T}(AA^{T}({\ddot e}+\omega_{0}^{2}e)+Be)\nonumber\\
&=&-{\dot e}^{T}G{\dot e}=-\sum_{r<s}g_{rs}({\dot e}_{r}-{\dot
e}_{s})^{2}\,.
\end{eqnarray}
Hence ${\dot W}\leq 0$, meaning $W$ is nonincreasing. As a result,
the solution $v(t)$ has to be bounded. Since produced by an LTI
system, $v(t)$ can be written as a sum (of finitely many terms)
\begin{eqnarray}\label{eqn:sum}
v(t)=\sum_{k}{\rm Re}\left(\pi_{k}(t)e^{\lambda_{k}t}\right)
\end{eqnarray}
where $\pi_{k}(t)$ are polynomials with vector coefficients and
$\lambda_{k}\in\Complex$ are distinct. That $v(t)$ is bounded
therefore implies ${\rm Re}(\lambda_{k})\leq 0$ for all $k$; and
when ${\rm Re}(\lambda_{k})=0$ for some $k$ the corresponding
polynomial $\pi_{k}(t)$ must be of degree zero, i.e., a constant
vector.
\end{remark}

\section{Linkage}

Some of the conditions (for synchronization) we present are of
structural nature and require the introduction of a graph-like
object associated to the network~\eqref{eqn:motion}. It is defined
as follows. Let $\V=\{\nu_{1},\,\nu_{2},\,\ldots,\,\nu_{n}\}$ be the
set of nodes of our network. Over this node set we construct two
(undirected) graphs. The first one is the oscillator graph
$(\V,\,\setE_{\rm o})$, where the edge set $\setE_{\rm o}$ contains
$q\geq 2$ unordered distinct pairs $\{\nu_{r},\,\nu_{s}\}\subset\V$
of nodes between which there is an oscillator in the network. In
other words, $\{\nu_{r},\,\nu_{s}\}\in\setE_{\rm o}$ if the
incidence matrix $A$ has a column that reads either
$A(:,\,k)=I(:,\,r)-I(:,\,s)$ or $A(:,\,k)=I(:,\,s)-I(:,\,r)$. Since
each node in the network is incident to at least one oscillator, the
graph $(\V,\,\setE_{\rm o})$ has no isolated nodes. The second graph
is the coupler graph $(\V,\,\setE_{\rm c})$, where
$\{\nu_{r},\,\nu_{s}\}\in\setE_{\rm c}$ when $g_{rs}+b_{rs}>0$. We
now combine these two graphs into a single object $(\V,\,\setE_{\rm
o},\,\setE_{\rm c})$, which we call the {\em linkage} of the
network~\eqref{eqn:nonstar}.

Recall that a graph $(\V,\,\setE)$ is {\em bipartite} if we can find
two disjoint sets of nodes $\V_{1},\,\V_{2}\subset\V$ such that
$\V_{1}\cup\V_{2}=\V$ and each edge $\epsilon\in\setE$ extends
between $\V_{1}$ and $\V_{2}$, i.e., we can write
$\epsilon=\{\nu_{r},\,\nu_{s}\}$ for some $\nu_{r}\in\V_{1}$ and
$\nu_{s}\in\V_{2}$. Such pair $(\V_{1},\,\V_{2})$ is called a {\em
bipartition} of $(\V,\,\setE)$, which is not necessarily unique. An
equivalent definition would have been the following: A graph is
bipartite if it has no odd cycles \cite{asratian98}. Below we
introduce a generalization of this.

\begin{definition}
The linkage $(\V,\,\setE_{\rm o},\,\setE_{\rm c})$ is said to be
{\em bipartite} if $\setE_{\rm o}\cap\setE_{\rm c}=\emptyset$ and no
cycle of the graph $(\V,\,\setE_{\rm o}\cup\setE_{\rm c})$ contains
an odd number of edges from $\setE_{\rm o}$.
\end{definition}

For instance, the linkage of the network in Fig.~\ref{fig:cross}a is
bipartite, whereas that of the network in Fig.~\ref{fig:cross}b is
not. The sister definition is given next.

\begin{definition}
The linkage $(\V,\,\setE_{\rm o},\,\setE_{\rm c})$ is said to be
{\em bilayer} if the graph $(\V,\,\setE_{\rm o})$ has a bipartition
$(\V_{1},\,\V_{2})$ such that no edge in $\setE_{\rm c}$ extends
between $\V_{1}$ and $\V_{2}$. When this holds there exist two
disjoint subgraphs $(\V_{1},\,\setE_{1})$ and
$(\V_{2},\,\setE_{2})$, called the {\em layers}, satisfying
$(\V_{1},\,\setE_{1})\cup(\V_{2},\,\setE_{2})=(\V,\,\setE_{\rm c})$.
\end{definition}

We now establish that bipartite and bilayer mean the same thing.

\begin{lemma}\label{lem:equivalence}
The linkage $(\V,\,\setE_{\rm o},\,\setE_{\rm c})$ is bipartite if
and only if it is bilayer.
\end{lemma}

\begin{proof}
Given the linkage $\L=(\V,\,\setE_{\rm o},\,\setE_{\rm c})$, let the
graph $\G=(\V,\,\setE_{\rm o}\cup\setE_{\rm c})$ has $\kappa$
components. This means we can find disjoint node sets
$\V^{1},\,\V^{2},\,\ldots,\,\V^{\kappa}\subset\V$ as well as edge
sets $\setE_{\rm o}^{1},\,\setE_{\rm o}^{2},\,\ldots,\,\setE_{\rm
o}^{\kappa}\subset\setE_{\rm o}$ and $\setE_{\rm c}^{1},\,\setE_{\rm
c}^{2},\,\ldots,\,\setE_{\rm c}^{\kappa}\subset\setE_{\rm c}$ such
that $\G=\bigcup_{i=1}^{\kappa}(\V^{i},\,\setE_{\rm
o}^{i}\cup\setE_{\rm c}^{i})$. It is not difficult to see that $\L$
is bipartite (bilayer) if and only if all the triples
$(\V^{i},\,\setE_{\rm o}^{i},\,\setE_{\rm c}^{i})$ are separately
bipartite (bilayer). Hence, without loss of generality, we
henceforth assume $\kappa=1$, i.e., $\G$ is connected.

{\em Part I: BP$\implies$BL.} Suppose $\L$ is bipartite. Let $\T$ be
a tree of $\G$. Now, using $\T$ we partition the node set
$\V=\{\nu_{1},\,\nu_{2},\,\ldots,\,\nu_{n}\}$ into two disjoint sets
$\V_{1},\,\V_{2}\subset\V$ as follows. We begin with letting
$\nu_{1}\in\V_{1}$. Then for $r=2,\,3,\,\ldots,\,n$ the node
$\nu_{r}$ is made belong to $\V_{1}$ if the (unique) path (in $\T$)
connecting $\nu_{1}$ to $\nu_{r}$ (denoted
$\setP_{\nu_{1}\to\nu_{r}}$) contains an even number of edges from
$\setE_{\rm o}$ (called {\em o-edges}). Once such construction of
$\V_{1}$ is complete, we let $\V_{2}=\V\setminus\V_{1}$. Now, our
claim is that the pair $(\V_{1},\,\V_{2})$ has to be a bipartition
of the graph $(\V,\,\setE_{\rm o})$. To see that, suppose otherwise.
That is, there exists an edge $\{\nu_{r},\,\nu_{s}\}\in\setE_{\rm
o}$ such that either $\nu_{r},\,\nu_{s}\in\V_{1}$ or
$\nu_{r},\,\nu_{s}\in\V_{2}$. Without loss of generality let both
$\nu_{r}$ and $\nu_{s}$ belong to $\V_{1}$, which implies both paths
$\setP_{\nu_{1}\to\nu_{r}}$ and $\setP_{\nu_{1}\to\nu_{s}}$ contain
an even number of o-edges. Let us consider the following two cases
separately. {\em Case 1: The edge $\{\nu_{r},\,\nu_{s}\}$ belongs to
the tree $\T$.} Note that we either have
$\setP_{\nu_{1}\to\nu_{s}}=\setP_{\nu_{1}\to\nu_{r}}\cup\setP_{\nu_{r}\to\nu_{s}}$
or
$\setP_{\nu_{1}\to\nu_{r}}=\setP_{\nu_{1}\to\nu_{s}}\cup\setP_{\nu_{s}\to\nu_{r}}$.
Without loss of generality let the former be the case. Then, since
$\{\nu_{r},\,\nu_{s}\}\in\setE_{\rm o}$, if the number of o-edges in
the path $\setP_{\nu_{1}\to\nu_{r}}$ is $p$ then the path
$\setP_{\nu_{1}\to\nu_{s}}$ must contain $p+1$ o-edges. Since $p$
and $p+1$ have different parities, it is not possible that both are
even. That is, $\nu_{r}$ and $\nu_{s}$ cannot both belong to
$\V_{1}$. Having ruled out this case let us now consider the other
scenario. {\em Case 2: The edge $\{\nu_{r},\,\nu_{s}\}$ does not
belong to the tree $\T$.} Then there is a unique cycle $\C$ of $\G$
with the following property. The cycle $\C$ contains the edge
$\{\nu_{r},\,\nu_{s}\}$ and all its other edges belong to the tree
$\T$. Let $\nu_{u}$ be the node of $\C$ such that the path
$\setP_{\nu_{1}\to\nu_{u}}$ is shorter than $\setP_{\nu_{1}\to\nu}$
for any other node $\nu$ in $\C$. Note that both paths
$\setP_{\nu_{1}\to\nu_{r}}$ and $\setP_{\nu_{1}\to\nu_{s}}$ should
pass from the vertex $\nu_{u}$, i.e., we have
\begin{subeqnarray}\label{eqn:path}
\setP_{\nu_{1}\to\nu_{r}}&=&\setP_{\nu_{1}\to\nu_{u}}\cup\setP_{\nu_{u}\to\nu_{r}}\\
\setP_{\nu_{1}\to\nu_{s}}&=&\setP_{\nu_{1}\to\nu_{u}}\cup\setP_{\nu_{u}\to\nu_{s}}\,.
\end{subeqnarray}
Let now $p_{1}$ and $p_{2}$ be the numbers of o-edges that
$\setP_{\nu_{u}\to\nu_{r}}$ and $\setP_{\nu_{u}\to\nu_{s}}$ contain,
respectively. Since both $\nu_{r}$ and $\nu_{s}$ belong to $\V_{1}$,
the numbers $p_{1}$ and $p_{2}$ must have the same parity by
\eqref{eqn:path}. Hence $p_{1}+p_{2}$ must be even. Note that
$\C=\setP_{\nu_{u}\to\nu_{r}}\cup\setP_{\nu_{r}\to\nu_{s}}\cup\setP_{\nu_{s}\to\nu_{u}}$.
Therefore (thanks to $\{\nu_{r},\,\nu_{s}\}\in\setE_{\rm o}$) the
number of o-edges in the cycle $\C$ must be $p_{1}+p_{2}+1$, which
is an odd number. But this contradicts the fact that $\L$ is
bipartite.

Having shown that $(\V_{1},\,\V_{2})$ is a bipartition of the graph
$(\V,\,\setE_{\rm o})$, we now establish the second condition for
$\L$ to be bilayer. That is, no edge in $\setE_{\rm c}$ (called {\em
c-edge}) extends between $\V_{1}$ and $\V_{2}$. The demonstration is
very similar to that of the first condition. Once again we employ
contradiction. Suppose there is an edge
$\{\nu_{r},\,\nu_{s}\}\in\setE_{\rm c}$ such that $\nu_{r}\in\V_{1}$
and $\nu_{s}\in\V_{2}$. Let us first study {\em Case 1: The edge
$\{\nu_{r},\,\nu_{s}\}$ belongs to the tree $\T$.} Then both paths
$\setP_{\nu_{1}\to\nu_{r}}$ and $\setP_{\nu_{1}\to\nu_{s}}$ must
contain the same number of o-edges, which implies either both
$\nu_{r}$ and $\nu_{s}$ belong to $\V_{1}$ or neither do. But this
cannot coexist with our initial assumption that $\nu_{r}\in\V_{1}$
and $\nu_{s}\in\V_{2}$. Now we consider {\em Case 2: The edge
$\{\nu_{r},\,\nu_{s}\}$ does not belong to the tree $\T$.} Then
there is a unique cycle $\C$ of $\G$ with the following property.
The cycle $\C$ contains the edge $\{\nu_{r},\,\nu_{s}\}$ and all its
other edges belong to the tree $\T$. Let $\nu_{u}$ be the node of
$\C$ such that the path $\setP_{\nu_{1}\to\nu_{u}}$ is shorter than
$\setP_{\nu_{1}\to\nu}$ for any other node $\nu$ in $\C$. Note that
both paths $\setP_{\nu_{1}\to\nu_{r}}$ and
$\setP_{\nu_{1}\to\nu_{s}}$ should pass from the vertex $\nu_{u}$,
i.e., we have \eqref{eqn:path}. Let now $p_{1}$ and $p_{2}$ be the
numbers of o-edges that $\setP_{\nu_{u}\to\nu_{r}}$ and
$\setP_{\nu_{u}\to\nu_{s}}$ contain, respectively. Since
$\nu_{r}\in\V_{1}$ and $\nu_{s}\in\V_{2}$, the numbers $p_{1}$ and
$p_{2}$ must have different parities by \eqref{eqn:path}. Hence
$p_{1}+p_{2}$ must be odd. Note that
$\C=\setP_{\nu_{u}\to\nu_{r}}\cup\setP_{\nu_{r}\to\nu_{s}}\cup\setP_{\nu_{s}\to\nu_{u}}$.
Therefore the number of o-edges in the cycle $\C$ must be
$p_{1}+p_{2}$, which is an odd number. But this contradicts the fact
that $\L$ is bipartite. To sum up, we have established that the pair
$(\V_{1},\,\V_{2})$ makes a bipartition of the graph
$(\V,\,\setE_{\rm o})$ and that no c-edge extends between $\V_{1}$
and $\V_{2}$. Hence the linkage $\L=(\V,\,\setE_{\rm o},\,\setE_{\rm
c})$ is bilayer. In the second part of the proof we show the other
direction.

{\em Part II: BL$\implies$BP.} Suppose $\L$ is bilayer. Then the
graph $(\V,\,\setE_{\rm o})$ has a bipartition $(\V_{1},\,\V_{2})$
such that no c-edge extends between $\V_{1}$ and $\V_{2}$. Let $\C$
be a cycle of the graph $\G$ with the node sequence
$(\nu_{r_{1}},\,\nu_{r_{2}},\,\nu_{r_{3}},\,\ldots,\,\nu_{r_{\alpha}},\,\nu_{r_{1}})$
where each consecutive pair of nodes
$\{\nu_{r_{\beta}},\,\nu_{r_{\beta+1}}\}$ (letting
$r_{\alpha+1}=r_{1}$) is an edge of $\G$. Note that
$\{\nu_{r_{\beta}},\,\nu_{r_{\beta+1}}\}$ is an o-edge if the
constituent nodes belong to different node subsets:
$\nu_{r_{\beta}}\in\V_{1}$ and $\nu_{r_{\beta+1}}\in\V_{2}$ or vice
versa. Otherwise, i.e., when
$\nu_{r_{\beta}},\,\nu_{r_{\beta+1}}\in\V_{1}$ or
$\nu_{r_{\beta}},\,\nu_{r_{\beta+1}}\in\V_{2}$, the edge
$\{\nu_{r_{\beta}},\,\nu_{r_{\beta+1}}\}$ is a c-edge. That the
sequence
$(\nu_{r_{1}},\,\nu_{r_{2}},\,\nu_{r_{3}},\,\ldots,\,\nu_{r_{\alpha}},\,\nu_{r_{1}})$
begins and ends with the same node has one obvious implication: the
number of transitions between $\V_{1}$ and $\V_{2}$ while one
traverses this sequence must always be even. Since the number of
transitions equals the number of o-edges, we conclude that the
number of o-edges in the cycle $\C$ is even. This is true for any
cycle because $\C$ was arbitrary. Therefore $\L$ is bipartite.
\end{proof}

\section{Purely resistive coupling}

In this section we focus on the special case where the LC
oscillators are coupled via resistors only. Namely, we study the
network~\eqref{eqn:nonstar} under $B=0$. Consider therefore
\begin{eqnarray}\label{eqn:nonstares}
AA^{T}({\ddot e}+\omega_{0}^{2}e)+G{\dot e}=0\,,\qquad v=A^{T}e\,.
\end{eqnarray}
We will show that whether the network~\eqref{eqn:nonstares} is
synchronous or not depends only on the structure of the coupling.
This would imply that two separate (resistively coupled) networks
are both synchronous if they share the same linkage. In
Section~\ref{sec:ex} we will see that this is not the case for the
general situation where inductive coupling is also present. The
following lemma is the first step toward the main result
(Theorem~\ref{thm:bl}) of this section.

\begin{lemma}\label{lem:necessary}
The network~\eqref{eqn:nonstares} is synchronous only when its
linkage is bipartite.
\end{lemma}

\begin{proof}
Suppose the network~\eqref{eqn:nonstares} is synchronous. Let
$v(t)=A^{T}e(t)$ be an arbitrary solution satisfying $v(t)\not\to
0$. Recall that $v(t)$ is bounded and satisfies \eqref{eqn:sum}.
Hence $v(t)\not\to 0$ implies that the sum~\eqref{eqn:sum} must
include a term of the form ${\rm Re}(\xi e^{j\omega t})$ for some
$\omega\in\Real$ and nonzero $\xi\in\Complex^{q}$. The network being
LTI, the mapping $t\mapsto{\rm Re}(\xi e^{j\omega t})$ must itself
be a solution. Hence, it might have been that $v(t)={\rm Re}(\xi
e^{j\omega t})$. We henceforth focus on this case. Observe that the
solution $v(t)={\rm Re}(\xi e^{j\omega t})$ is periodic;
consequently, so is the corresponding
$W(t)=\frac{1}{2}\omega_{0}^{2}v(t)^{T}v(t)+\frac{1}{2}{\dot
v}(t)^{T}{\dot v}(t)$. Recall that $W(t)$ is also nonincreasing; see
Remark~\ref{rem:passive}. Hence $W(t)$ must be constant, which
yields $0={\dot W}=-{\dot e}^{T}G{\dot e}$ by \eqref{eqn:Wdot}.
Since $G$ is symmetric positive semidefinite, ${\dot e}^{T}G{\dot
e}=0$ implies $G{\dot e}=0$. Under this condition the
equation~\eqref{eqn:nonstares} simplifies into $AA^{T}({\ddot
e}+\omega_{0}^{2}e)=0$. Multiplying this from left with $({\ddot
e}+\omega_{0}^{2}e)^{T}$ allows us to write $0=({\ddot
e}+\omega_{0}^{2}e)^{T}AA^{T}({\ddot
e}+\omega_{0}^{2}e)=\|A^{T}({\ddot
e}+\omega_{0}^{2}e)\|^{2}=\|{\ddot v}+\omega_{0}^{2}v\|^{2}$. That
is, our choice $v(t)={\rm Re}(\xi e^{j\omega t})$ satisfies ${\ddot
v}+\omega_{0}^{2}v=0$. This at once brings $\omega=\omega_{0}$.
Furthermore, since the network~\eqref{eqn:nonstares} is synchronous,
we should have $v(t)=[v_{1}(t)\ v_{2}(t)\ \cdots\ v_{q}(t)]^{T}$
with $v_{k}(t)=a_{k}\mu\sin(\omega_{0}t+\phi)$, where $\mu>0$,
$\phi\in[0,\,2\pi)$, and $a_{k}\in\{-1,\,1\}$. Without loss of
generality we can let $\mu=1$ and $\phi=0$. Hence we have
established the following. The synchronous
network~\eqref{eqn:nonstares} admits a pair $(v(t),\,e(t))$ such
that
\begin{eqnarray}\label{eqn:ve}
v(t)=[a_{1}\ a_{2}\ \cdots\
a_{q}]^{T}\times\sin(\omega_{0}t)\,,\qquad G{\dot e}(t)=0
\end{eqnarray}
where $a_{k}\in\{-1,\,1\}$ for $k=1,\,2,\,\ldots,\,q$.

Using \eqref{eqn:ve} we now show that the linkage $(\V,\,\setE_{\rm
o},\,\setE_{\rm c})$ is bipartite. To this end, let us first
establish $\setE_{\rm o}\cap\setE_{\rm c}=\emptyset$. Suppose
otherwise. That is, there exists an edge
$\{\nu_{r},\,\nu_{s}\}\in\setE_{\rm o}\cap\setE_{\rm c}$. This means
there are an oscillator (say the $k$th oscillator) and a resistor
(with conductance $g_{rs}>0$) in the network both extending between
the nodes $\nu_{r}$ and $\nu_{s}$. Then we either have
$v_{k}=e_{r}-e_{s}$ or $v_{k}=e_{s}-e_{r}$. Without loss of
generality let $v_{k}=e_{r}-e_{s}$. Since $G{\dot e}=0$ we have
${\dot e}_{r}(t)-{\dot e}_{s}(t)\equiv 0$ by \eqref{eqn:Wdot}. This
however contradicts
\begin{eqnarray*}
{\dot e}_{r}(t)-{\dot e}_{s}(t)={\dot v}_{k}(t)
=a_{k}\omega_{0}\cos(\omega_{0}t) \not\equiv0\,.
\end{eqnarray*}

The second thing we have to show is that no cycle of the graph
$(\V,\,\setE_{\rm o}\cup\setE_{\rm c})$ contains an odd number of
edges from $\setE_{\rm o}$. Again suppose otherwise. Then there is a
cycle with the set of edges
$\setS=\{\{\nu_{r_{1}},\,\nu_{r_{2}}\},\,\{\nu_{r_{2}},\,\nu_{r_{3}}\},\,\ldots,\,\{\nu_{r_{p-1}},\,\nu_{r_{p}}\},\,\{\nu_{r_{p}},\,\nu_{r_{1}}\}\}
\subset\setE_{\rm o}\cup\setE_{\rm c}$, comprising $p$ edges, an odd
number $2m+1\leq p$ of which belongs to $\setE_{\rm o}$. Without
loss of generality let $\{\nu_{r_{p}},\,\nu_{r_{1}}\}\in\setE_{\rm
o}$. Consider the identity
\begin{eqnarray*}
({\dot e}_{r_{1}}-{\dot e}_{r_{p}})=({\dot e}_{r_{1}}-{\dot
e}_{r_{2}})+({\dot e}_{r_{2}}-{\dot e}_{r_{3}})+\cdots+({\dot
e}_{r_{p-1}}-{\dot e}_{r_{p}})\,.
\end{eqnarray*}
Let us remove from the righthand side the terms $({\dot
e}_{r_{\ell}}-{\dot e}_{r_{\ell+1}})$ for which
$\{\nu_{r_{\ell}},\,\nu_{r_{\ell+1}}\}\in\setE_{\rm c}$. We can do
this because such terms satisfy $({\dot e}_{r_{\ell}}-{\dot
e}_{r_{\ell+1}})=0$ thanks to $G{\dot e}=0$ and \eqref{eqn:Wdot}.
Then the above equality simplifies into
\begin{eqnarray}\label{eqn:Eosc}
({\dot e}_{r_{1}}-{\dot
e}_{r_{p}})=\sum_{\{\nu_{r_{\ell}},\,\nu_{r_{\ell+1}}\}\in\setS\cap\setE_{\rm
o}}({\dot e}_{r_{\ell}}-{\dot e}_{r_{\ell+1}})\,.
\end{eqnarray}
Since each term in \eqref{eqn:Eosc} corresponds to some edge in
$\setE_{\rm o}$, we can find indices
$k_{0},\,k_{1},\,\ldots,\,k_{2m}\in\{1,\,2,\,\ldots,\,q\}$ and
coefficients (determined by the polarities of the oscillator
voltages) $c_{0},\,c_{1},\,\ldots,\,c_{2m}\in\{-1,\,1\}$ that allow
us to rewrite \eqref{eqn:Eosc} as
\begin{eqnarray}\label{eqn:akbk}
c_{k_{0}}{\dot v}_{k_{0}}=c_{k_{1}}{\dot v}_{k_{1}}+c_{k_{2}}{\dot
v}_{k_{2}}+\cdots+c_{k_{2m}}{\dot v}_{k_{2m}}\,.
\end{eqnarray}
Combining \eqref{eqn:akbk} and \eqref{eqn:ve} then implies
$a_{k_{0}}c_{k_{0}}=a_{k_{1}}c_{k_{1}}+a_{k_{2}}c_{k_{2}}+\cdots+a_{k_{2m}}c_{k_{2m}}$
which never admits a solution because the lefthand side is always
odd, while the righthand side is even. The result hence follows by
contradiction.
\end{proof}

\begin{theorem}\label{thm:bl}
The network~\eqref{eqn:nonstares} is synchronous if and only if the
linkage $(\V,\,\setE_{\rm o},\,\setE_{\rm c})$ is bilayer and both
its layers are connected.
\end{theorem}

\begin{proof}
Let the network~\eqref{eqn:nonstares} be synchronous. Then, by
Lemma~\ref{lem:equivalence} and Lemma~\ref{lem:necessary}, its
linkage $(\V,\,\setE_{\rm o},\,\setE_{\rm c})$ is bilayer. Let the
graphs $(\V_{1},\,\setE_{1})$ and $(\V_{2},\,\setE_{2})$ be a pair
of associated layers. Recall that the pair $(\V_{1},\,\V_{2})$ is
then a bipartition of the graph $(\V,\,\setE_{\rm o})$ and we have
$(\V_{1},\,\setE_{1})\cup(\V_{2},\,\setE_{2})=(\V,\,\setE_{\rm c})$.
Let $n_{1}$ and $n_{2}$ denote the number of nodes in $\V_{1}$ and
$\V_{2}$, respectively. Suppose now one of the layers, say
$(\V_{1},\,\setE_{1})$, is not connected. Then we can find disjoint
subgraphs $(\V_{1}^{\prime},\,\setE_{1}^{\prime})$ and
$(\V_{1}^{\prime\prime},\,\setE_{1}^{\prime\prime})$ satisfying
$(\V_{1}^{\prime},\,\setE_{1}^{\prime})\cup(\V_{1}^{\prime\prime},\,\setE_{1}^{\prime\prime})=(\V_{1},\,\setE_{1})$.
Let $n_{1}^{\prime}$ and $n_{1}^{\prime\prime}$ denote the number of
nodes in $\V_{1}^{\prime}$ and $\V_{1}^{\prime\prime}$,
respectively. Without loss of generality (see
Remark~\ref{rem:polarity}) we can suppose that the Laplacian matrix
$G$ and the incidence matrix $A$ of the
network~\eqref{eqn:nonstares} have the following structures
\begin{eqnarray*}
G=\left[\begin{array}{ccc}G_{1}^{\prime}&0&0 \\
0&G_{1}^{\prime\prime}&0 \\ 0&0&G_{2}\end{array}\right]\,,\qquad
A=\left[\begin{array}{rr}F_{1}^{\prime}&0 \\
0&F_{1}^{\prime\prime} \\
-F_{2}^{\prime}&-F_{2}^{\prime\prime}\end{array}\right]
\end{eqnarray*}
where $G_{1}^{\prime}\in\Real^{n_{1}^{\prime}\times
n_{1}^{\prime}}$,
$G_{1}^{\prime\prime}\in\Real^{n_{1}^{\prime\prime}\times
n_{1}^{\prime\prime}}$, and $G_{2}\in\Real^{n_{2}\times n_{2}}$ are
all separately Laplacian matrices; and
$F_{1}^{\prime}\in\Real^{n_{1}^{\prime}\times q_{1}}$,
$F_{1}^{\prime\prime}\in\Real^{n_{1}^{\prime\prime}\times q_{2}}$,
and $[F_{2}^{\prime}\
F_{2}^{\prime\prime}]=F_{2}\in\Real^{n_{2}\times q}$ are class $\F$
matrices. Consider now the case where the first $n_{1}^{\prime}$
node voltages in the network read $\sin(\omega_{0}t)$ while the
remaining are fixed at zero. The corresponding node voltage vector
is $e(t)=[\one_{n_{1}^{\prime}}^{T}\ \ 0_{1\times
n_{1}^{\prime\prime}}\ \ 0_{1\times
n_{2}}]^{T}\times\sin(\omega_{0}t)$. Note that
\begin{eqnarray*}
G{\dot
e}(t)=\left[\begin{array}{ccc}G_{1}^{\prime}&0&0 \\
0&G_{1}^{\prime\prime}&0 \\
0&0&G_{2}\end{array}\right]\left[\begin{array}{c}\one_{n_{1}^{\prime}}
\\ 0\\ 0\end{array}\right]\omega_{0}\cos(\omega_{0}t)=\left[\begin{array}{c}G_{1}^{\prime}\one_{n_{1}^{\prime}}
\\ 0\\ 0\end{array}\right]\omega_{0}\cos(\omega_{0}t)=0
\end{eqnarray*}
because $G_{1}^{\prime}\one_{n_{1}^{\prime}}=0$. Note also that
${\ddot e}(t)+\omega_{0}^{2}e(t)=0$. Therefore $e(t)$ satisfies
\eqref{eqn:nonstares}, meaning the resulting oscillator voltage
vector $A^{T}e(t)=v(t)=[v_{1}(t)\ v_{2}(t)\ \cdots\ v_{q}(t)]^{T}$
is a possible solution of the network. We can write
\begin{eqnarray*}
v(t)=A^{T}e(t)=\left[\begin{array}{rrr}F_{1}^{\prime
T}&0&-F_{2}^{\prime T}\\ 0&F_{1}^{\prime\prime
T}&-F_{2}^{\prime\prime
T}\end{array}\right]\left[\begin{array}{c}\one_{n_{1}^{\prime}}
\\ 0\\
0\end{array}\right]\sin(\omega_{0}t)=\left[\begin{array}{c}F_{1}^{\prime
T}\one_{n_{1}^{\prime}}
\\ 0\end{array}\right]\sin(\omega_{0}t)=\left[\begin{array}{c}\one_{q_{1}}
\\ 0\end{array}\right]\sin(\omega_{0}t)
\end{eqnarray*}
from which we obtain
\begin{eqnarray}\label{eqn:vk}
v_{k}(t)=\left\{\begin{array}{cl}\sin(\omega_{0}t)&\mbox{for}\
k\in\{1,\,\ldots,\,q_{1}\}\,,\\\vspace{-0.1in} \\ 0 & \mbox{for}\
k\in\{q_{1}+1,\,\ldots,\,q\}\,.\end{array}\right.
\end{eqnarray}
Clearly, \eqref{eqn:vk} contradicts that the
network~\eqref{eqn:nonstares} is synchronous. Hence both layers
$(\V_{1},\,\setE_{1})$ and $(\V_{2},\,\setE_{2})$ must be connected.

Now we show the other direction. This time we start with assuming
the linkage $(\V,\,\setE_{\rm o},\,\setE_{\rm c})$ is bilayer and
both its layers are connected. Therefore, without loss of generality
(see Remark~\ref{rem:polarity}), we can let
\begin{eqnarray*}
G=\left[\begin{array}{cc}G_{1}&0 \\
0&G_{2}\end{array}\right]\,,\qquad
A=\left[\begin{array}{r}F_{1} \\
-F_{2}\end{array}\right]
\end{eqnarray*}
where each of the Laplacian matrices $G_{1}\in\Real^{n_{1}\times
n_{1}}$ and $G_{2}\in\Real^{n_{2}\times n_{2}}$ represents one of
the layers and $F_{1}\in\Real^{n_{1}\times q}$ and
$F_{2}\in\Real^{n_{2}\times q}$ are class-$\F$ matrices. Since the
layers are connected the ranks of the matrices $G_{1}$ and $G_{2}$
are $n_{1}-1$ and $n_{2}-1$, respectively. In particular, we have
${\rm null}\,G_{1}={\rm span}\{\one_{n_{1}}\}$ and ${\rm
null}\,G_{2}={\rm span}\{\one_{n_{2}}\}$. Suppose now the
network~\eqref{eqn:nonstares} is not synchronous. This means (see
the proof of Lemma~\ref{lem:necessary}) either of the following two
cases must take place.

{\em Case~1: There exists a solution of the form $v(t)={\rm Re}(\xi
e^{j\omega_{0}t})$ with nonzero $\xi=[a_{1}\ a_{2}\ \cdots\
a_{q}]^{T}$ satisfying $a_{k}/a_{\ell}\notin\{-1,\,1\}$ for some
$(k,\,\ell)$.} Let $e(t)={\rm Re}(\zeta e^{j\omega_{0}t})$ be the
corresponding node voltage vector, i.e., $A^{T}e(t)=v(t)$. Then let
$\zeta_{1}\in\Complex^{n_{1}}$ and $\zeta_{2}\in\Complex^{n_{2}}$
satisfy $[\zeta_{1}^{T}\ \zeta_{2}^{T}]^{T}=\zeta$. Since $G{\dot
e}(t)=0$ (see the proof of Lemma~\ref{lem:necessary}) we have to
have $G_{1}\zeta_{1}=0$ and $G_{2}\zeta_{2}=0$. Also,
$A^{T}e(t)=v(t)$ implies
$F_{1}^{T}\zeta_{1}-F_{2}^{T}\zeta_{2}=\xi$. Observe that since
$\zeta_{1}\in{\rm span}\{\one_{n_{1}}\}$ and $F_{1}$ is a class-$\F$
matrix, we have $F_{1}^{T}\zeta_{1}\in{\rm span}\{\one_{q}\}$.
Similarly, $F_{2}^{T}\zeta_{2}\in{\rm span}\{\one_{q}\}$. As a
result $\xi\in{\rm span}\{\one_{q}\}$. But this violates our initial
assumption on $\xi$.

{\em Case~2: All solutions satisfy $v(t)\to 0$.} To show that this
case, too, is impossible we construct a counterexample solution
which does not vanish as $t\to\infty$. Let the first $n_{1}$ node
voltages in the network read $\sin(\omega_{0}t)$ while the remaining
are fixed at zero. That is, $e(t)=[\one_{n_{1}}^{T}\ 0_{1\times
n_{2}}]^{T}\times\sin(\omega_{0}t)$. This allows us to write
\begin{eqnarray*}
G{\dot
e}(t)=\left[\begin{array}{cc}G_{1}&0 \\
0&G_{2}\end{array}\right]\left[\begin{array}{c}\one_{n_{1}}
\\ 0\end{array}\right]\omega_{0}\cos(\omega_{0}t)=\left[\begin{array}{c}G_{1}\one_{n_{1}} \\ 0\end{array}\right]\omega_{0}\cos(\omega_{0}t)=0
\end{eqnarray*}
because $G_{1}\one_{n_{1}}=0$. Note also that ${\ddot
e}(t)+\omega_{0}^{2}e(t)=0$. Therefore $e(t)$ satisfies
\eqref{eqn:nonstares}, meaning the resulting oscillator voltage
vector $v(t)=A^{T}e(t)$ is a possible solution of the network. Since
$F_{1}^{T}\one_{n_{1}}=\one_{q}$ we can proceed as
\begin{eqnarray*}
v(t)=A^{T}e(t)=[F_{1}^{T}\
-F_{2}^{T}]\left[\begin{array}{c}\one_{n_{1}}\\
0\end{array}\right]\sin(\omega_{0}t)=F_{1}^{T}\one_{n_{1}}\sin(\omega_{0}t)=\one_{q}\sin(\omega_{0}t)\,.
\end{eqnarray*}
Clearly, $v(t)\not\to 0$. Both cases are now ruled out. Hence the
network~\eqref{eqn:nonstares} has to be synchronous.
\end{proof}

\section{Generalized eigenvalues}\label{sec:reig}

In the previous section we have discovered it is necessary that the
linkage of the network~\eqref{eqn:nonstares} is bilayer for
synchronous behavior. In the remainder of the paper we will assume
this condition for the general framework~\eqref{eqn:nonstar}. Also,
for simplicity of analysis, we will consider the type of networks
where the subspace that contains the oscillator voltage vector
$v=A^{T}e$ is the entire space, i.e., ${\rm
range}(A^{T})=\Real^{q}$. To sum up, we henceforth make

\begin{assumption}\label{assume:bilayer}
The following hold.
\begin{itemize}
\item The linkage of the network~\eqref{eqn:nonstar} is bilayer.
\item ${\rm rank}(A)=q$.
\end{itemize}
\end{assumption}

Let $P,Q\in\Complex^{n\times n}$. Recall that a generalized
eigenvalue $\lambda\in\Complex$ of the pair $(P,\,Q)$ satisfies
$(P-\lambda Q)x=0$ for some $x\neq 0$, where $x\in\Complex^{n}$.
Observe that if $(P,\,Q)$ is a Laplacian pair, which implies
$P\one_{n}=Q\one_{n}=0$, then the corresponding set of eigenvalues
according to this definition is the entire complex plane $\Complex$.
To avoid this kind of outrage, we introduce a slightly modified
version.

\begin{definition}
Let $P,Q\in\Complex^{n\times n}$. A {\em restricted generalized
eigenvalue} $\lambda\in\Complex$ of the pair $(P,\,Q)$ satisfies
\begin{eqnarray*}
(P-\lambda Q)x=0\quad\mbox{for some}\quad Qx\neq0
\end{eqnarray*}
where $x\in\Complex^{n}$. The set of all restricted generalized
eigenvalues is denoted by ${\rm reig}(P,\,Q)$.
\end{definition}

In this section we study the problem of computing ${\rm
reig}(G+jB,\,AA^{T})$ for the network~\eqref{eqn:nonstar}. To this
end, we propose a method to reduce this generalized eigenvalue
problem to a standard eigenvalue problem. Then, in the next section,
we uncover the link between these generalized eigenvalues and
network synchronization.

Consider the network~\eqref{eqn:nonstar} under
Assumption~\ref{assume:bilayer}. Since the linkage of the network is
bilayer we will henceforth let, without loss of generality (see
Remark~\ref{rem:polarity}), the matrices $G,\,B\in\Real^{n\times n}$
and $A\in\Real^{n\times q}$ have the following structures
\begin{eqnarray}\label{eqn:forms}
G=\left[\begin{array}{cc}G_{1}&0 \\
0&G_{2}\end{array}\right]\,,\qquad
B=\left[\begin{array}{cc}B_{1}&0 \\
0&B_{2}\end{array}\right]\,,\qquad
A=\left[\begin{array}{r}F_{1} \\
-F_{2}\end{array}\right]
\end{eqnarray}
where the submatrices $G_{1},\,B_{1}\in\Real^{n_{1}\times n_{1}}$
and $G_{2},\,B_{2}\in\Real^{n_{2}\times n_{2}}$ are all Laplacians
and $F_{1}\in\Real^{n_{1}\times q}$ and $F_{2}\in\Real^{n_{2}\times
q}$ are class $\F$ matrices. We now investigate the properties of
the solution $(E,\,Y)$ of the following equation
\begin{eqnarray}\label{eqn:Lambda}
\underbrace{\left[\begin{array}{cc} G+jB & -A\\A^{T}& 0
\end{array}\right]}_{M}\left[\begin{array}{c} E \\ Y
\end{array}\right]=\underbrace{\left[\begin{array}{c} 0 \\ I
\end{array}\right]}_{N}
\end{eqnarray}
where $E\in\Complex^{n\times q}$ and $Y\in\Complex^{q\times q}$.

\begin{theorem}\label{thm:Lambda}
The equation~\eqref{eqn:Lambda} admits a solution $(E,\,Y)$ with a
unique $Y$. Moreover, the following hold.
\begin{enumerate}
\item $Y=Y^{T}$.
\item Each eigenvalue $\lambda$ of $Y$ satisfies ${\rm Re}(\lambda)\geq
0$ and ${\rm Im}(\lambda)\geq 0$.
\item $Y\one_{q}=0$
\item If $B=0$ then $Y$ is real and $Y\geq 0$.
\end{enumerate}
\end{theorem}

\begin{proof}
{\em Existence.} A solution $(E,\,Y)$ exists for \eqref{eqn:Lambda}
if and only if ${\rm range}(M)\supset{\rm range}(N)$, which is
equivalent to
\begin{eqnarray}\label{eqn:establish}
{\rm null}(M^{*})\subset {\rm null}(N^{*})\,.
\end{eqnarray}
To establish \eqref{eqn:establish} suppose otherwise, that is, ${\rm
null}(M^{*})\not\subset{\rm null}(N^{*})$. This implies that we can
find a nonzero vector $\eta$ satisfying $M^{*}\eta=0$ and
$N^{*}\eta\neq 0$. Let us partition this vector as
$\eta=[\eta_{1}^{T}\ \eta_{2}^{T}]^{T}$ where
$\eta_{1}\in\Complex^{n}$ and $\eta_{2}\in\Complex^{q}$. Expanding
$M^{*}\eta=0$ gives us
\begin{eqnarray}\label{eqn:eta1}
(G-jB)\eta_{1}+A\eta_{2}&=&0\\
-A^{T}\eta_{1}&=&0\,.\nonumber
\end{eqnarray}
which allows us to write
\begin{eqnarray*}
\eta_{1}^{*}G\eta_{1}-j\eta_{1}^{*}B\eta_{1}
=\eta_{1}^{*}(G-jB)\eta_{1} =-\eta_{1}^{*}A\eta_{2}
=(-A^{T}\eta_{1})^{*}\eta_{2} =0\,.
\end{eqnarray*}
Since the Laplacians $G$ and $B$ are symmetric positive semidefinite
we have $\eta_{1}^{*}G\eta_{1}=0$ and $\eta_{1}^{*}B\eta_{1}=0$,
which implies $G\eta_{1}=0$ and $B\eta_{1}=0$. Revisiting
\eqref{eqn:eta1} with this information lets us see $A\eta_{2}=0$,
whence we infer $\eta_{2}=0$ because $A$ is full column rank by
Assumption~\ref{assume:bilayer}. But this contradicts $N^{*}\eta\neq
0$ because $N^{*}\eta=\eta_{2}$.

{\em Uniqueness \& symmetry.} Let the pairs $(E_{1},\,Y_{1})$ and
$(E_{2},\,Y_{2})$ both satisfy \eqref{eqn:Lambda}. We can write
\begin{eqnarray}\label{eqn:uniqueness}
Y_{1}&=&[0\ \ I]\left[\begin{array}{r}-E_{1}\\ Y_{1}\end{array}\right]\nonumber\\
&=&[E_{2}^{T}\ Y_{2}^{T}]\left[\begin{array}{cc}G+jB&A\\
-A^{T}&0\end{array}\right]\left[\begin{array}{r}-E_{1}\\ Y_{1}\end{array}\right]\nonumber\\
&=&[E_{2}^{T}\ Y_{2}^{T}]\left[\begin{array}{c}0\\ I\end{array}\right]\nonumber\\
&=&Y_{2}^{T}\,.
\end{eqnarray}
The choice $(E_{1},\,Y_{1})=(E_{2},\,Y_{2})$ gives us at once the
symmetry $Y_{1}=Y_{1}^{T}$. Then, thanks to this symmetry,
\eqref{eqn:uniqueness} implies the uniqueness $Y_{1}=Y_{2}$.

{\em Eigenvalues.} Let $\lambda\in\Complex$ be an eigenvalue of $Y$
and $v\in\Complex^{q}$ be the corresponding unit eigenvector, i.e.,
$Y v=\lambda v$ and $v^{*}v=1$. Multiplying both sides of
\eqref{eqn:Lambda} by $v$ from right and letting $e=Ev$ we obtain
\begin{eqnarray*}
(G+jB)e-\lambda Av&=&0\\
A^{T}e&=&v\,.
\end{eqnarray*}
Using these identities we can write $\lambda=\lambda v^{*}v
=\lambda(e^{*}A)v =e^{*}(\lambda Av) =e^{*}Ge+je^{*}Be$. Since $G$
and $B$ are symmetric positive semidefinite, it follows that ${\rm
Re}(\lambda)\geq 0$ and ${\rm Im}(\lambda)\geq 0$.

{\em Null space.} Let $u\in\Real^{n}$ be $u=[0_{1\times n_{1}}\
\one_{n_{2}}^{T}]^{T}$. Using \eqref{eqn:forms} and the identities
$G_{2}\one_{n_{2}}=0$, $B_{2}\one_{n_{2}}=0$,
$F_{2}^{T}\one_{n_{2}}=\one_{q}$ we can write
\begin{eqnarray*}
\left[\begin{array}{c}0\\
\one_{q}\end{array}\right]=\left[\begin{array}{c}0_{n_{1}\times 1} \\ 0_{n_{2}\times 1}\\
\one_{q}\end{array}\right]=\left[\begin{array}{ccc}G_{1}+jB_{1}&0&F_{1}\\
0&G_{2}+jB_{2}&-F_{2}\\
-F_{1}^{T}&F_{2}^{T}&0\end{array}\right]\left[\begin{array}{c}0 \\
\one_{n_{2}}
\\ 0\end{array}\right]=\left[\begin{array}{cc}G+jB&A\\
-A^{T}&0\end{array}\right]\left[\begin{array}{c}u
\\
0\end{array}\right]\,.
\end{eqnarray*}
Then we have
\begin{eqnarray*}
Y \one_{q}&=&[E^{T}\
Y]\left[\begin{array}{c}0\\\one_{q}\end{array}\right]\\
&=&[E^{T}\
Y]\left[\begin{array}{cc}G+jB&A\\
-A^{T}&0\end{array}\right]\left[\begin{array}{c}u
\\
0\end{array}\right]\\
&=&[0\ \ I]\left[\begin{array}{c}u
\\
0\end{array}\right]\\
&=&0\,.
\end{eqnarray*}
thanks to \eqref{eqn:Lambda} and the symmetry of $Y$.

{\em Positive semidefiniteness.} Let $B=0$. Now that the matrix $M$
is real, a real solution $(E,\,Y)$ exists for \eqref{eqn:Lambda}.
Then $Y\in\Real^{q\times q}$ by uniqueness. To show that $Y\geq 0$
let us write (for real $E$)
\begin{eqnarray*}
Y&=&[E^{T}\ Y^{T}]\left[\begin{array}{c}0\\
I\end{array}\right]\\
&=&[E^{T}\ Y^{T}]\left[\begin{array}{cc} G & -A\\A^{T}& 0
\end{array}\right]\left[\begin{array}{c}E\\
Y\end{array}\right]\\
&=&E^{T}GE
\end{eqnarray*}
whence the positive semidefiniteness follows by $G\geq0$.
\end{proof}
\vspace{0.12in}

Theorem~\ref{thm:Lambda} motivates the following definition.

\begin{definition}
The {\em effective Laplacian} of the network~\eqref{eqn:nonstar} is
defined as
\begin{eqnarray*}
Y=[0\ \ I]\left[\begin{array}{cc} G+jB & -A\\A^{T}& 0
\end{array}\right]^{+}\left[\begin{array}{c}0 \\ I
\end{array}\right]
\end{eqnarray*}
which satisfies the properties 1-4 listed in
Theorem~\ref{thm:Lambda}.
\end{definition}

\begin{remark}
When $n_{1}=n_{2}$ and $F_{1}=F_{2}=I$, the effective Laplacian $Y$
equals the {\em parallel sum} of the matrices $Y_{1}=G_{1}+jB_{1}$
and $Y_{2}=G_{2}+jB_{2}$, i.e., $Y=Y_{1}(Y_{1}+Y_{2})^{+}Y_{2}$; see
\cite{berkics17}.
\end{remark}

\begin{remark}
The solution $v(t)$ of the resistively coupled
network~\eqref{eqn:nonstares} satisfies ${\ddot
v+\omega_{0}^{2}v}+Y{\dot v}=0$.
\end{remark}

We end this section with the observation that the generalized
eigenvalues of the pair $(G+jB,\,AA^{T})$ coincide with the
eigenvalues of the effective Laplacian of the network. This will
allow us to work with $Y$ instead of the pair $(G+jB,\,AA^{T})$ in
the next section.

\begin{theorem}
${\rm reig}(G+jB,\,AA^{T})={\rm eig}(Y)$.
\end{theorem}

\begin{proof}
Let $\lambda\in{\rm eig}(Y)$ and $v\in\Complex^{q}$ be the
corresponding eigenvector, i.e., $Y v=\lambda v$. Let $e=Ev$ where
$E$ satisfies \eqref{eqn:Lambda}. Now, multiplying both sides of
\eqref{eqn:Lambda} from right by $v$ and using $Y v=\lambda v$ and
$e=Ev$ we obtain
\begin{eqnarray*}
\left[\begin{array}{cc} G+jB & -A\\A^{T}& 0
\end{array}\right]\left[\begin{array}{c} e \\ \lambda v
\end{array}\right]=\left[\begin{array}{c} 0 \\ v
\end{array}\right]\,.
\end{eqnarray*}
Hence we can write $(G+jB)e-\lambda Av=0$ and $v=A^{T}e$.
Substituting $v$ in the first equation by $A^{T}e$ yields
$(G+jB-\lambda AA^{T})e=0$. Being an eigenvector, $v\neq 0$. As a
result $A^{T}e\neq 0$ which implies $AA^{T}e\neq 0$. This means
$\lambda\in{\rm reig}(G+jB,\,AA^{T})$. Therefore ${\rm
reig}(G+jB,\,AA^{T})\supset{\rm eig}(Y)$.

We now show the other direction. Let $\lambda\in{\rm
reig}(G+jB,\,AA^{T})$. Then we can find $e\in\Complex^{n}$
satisfying $(G+jB-\lambda AA^{T})e=0$ and $AA^{T}e\neq 0$. The
latter implies $A^{T}e\neq 0$. Letting the nonzero vector $v=A^{T}e$
we can cast $(G+jB-\lambda AA^{T})e=0$ into
\begin{eqnarray}\label{eqn:cast}
\left[\begin{array}{cc} G+jB & A\\-A^{T}& 0
\end{array}\right]\left[\begin{array}{c} -e \\ \lambda v
\end{array}\right]=\left[\begin{array}{c} 0 \\ v
\end{array}\right]\,.
\end{eqnarray}
Using \eqref{eqn:cast}, \eqref{eqn:Lambda}, and the symmetry of $Y$
we can write
\begin{eqnarray*}
Y v&=&[E^{T}\ Y]\left[\begin{array}{c} 0 \\ v
\end{array}\right]\\
&=&[E^{T}\ Y]\left[\begin{array}{cc} G+jB & A\\-A^{T}& 0
\end{array}\right]\left[\begin{array}{c} -e \\ \lambda v
\end{array}\right]\\
&=&\left(\left[\begin{array}{cc} G+jB & -A\\A^{T}& 0
\end{array}\right]\left[\begin{array}{c} E \\ Y
\end{array}\right]\right)^{T}\left[\begin{array}{c} -e \\ \lambda v
\end{array}\right]\\
&=&[0\ \ I]\left[\begin{array}{c} -e \\ \lambda v
\end{array}\right]\\
&=&\lambda v\,.
\end{eqnarray*}
Therefore ${\rm reig}(G+jB,\,AA^{T})\subset{\rm eig}(Y)$. The result
then follows.
\end{proof}

\section{RL coupling}

In Section~\ref{sec:reig} we considered an eigenvalue problem
focusing on the generalized eigenvalues of the pair
$(G+jB,\,AA^{T})$ associated to the network~\eqref{eqn:nonstar} that
has $q$ oscillators. We introduced a $q\times q$ matrix $Y$ called
the effective Laplacian of the network and established that the
above mentioned generalized eigenvalues coincide with the
eigenvalues of $Y$. Our analysis there was a preparation necessary
for our present investigation of the collective behavior of the
coupled oscillators of the network~\eqref{eqn:nonstar}. Now we will
show that whether the oscillators asymptotically synchronize or not
can be determined through the spectrum of the effective Laplacian.
As before, here, too, we posit Assumption~\ref{assume:bilayer} holds
and the matrices $G,\,B,\,A$ satisfy \eqref{eqn:forms}.

\begin{remark}\label{rem:nominal}
It is not difficult to see that $e(t)=[\one_{n_{1}}^{T}\
0_{n_{2}\times 1}]^{T}\times\sin(\omega_{0}t)$ and
$v(t)=\one_{q}\sin(\omega_{0}t)$ satisfy \eqref{eqn:nonstar}.
\end{remark}

Below is our main result.

\begin{theorem}\label{thm:sync}
The network~\eqref{eqn:nonstar} is synchronous if and only if the
effective Laplacian $Y$ has a single eigenvalue on the imaginary
axis.
\end{theorem}

We prove this theorem in two steps.

\begin{lemma}\label{lem:rolls}
The network~\eqref{eqn:nonstar} is not synchronous if and only if
there exist $\omega\in\Real$, ${\bar e}\in\Complex^{n}$, and ${\bar
v}\in\Complex^{q}\setminus{\rm span}\{\one_{q}\}$ satisfying
\begin{subeqnarray}\label{eqn:ebar}
((\omega_{0}^{2}-\omega^{2})AA^{T}+B){\bar e}&=&0\\
G{\bar e}&=&0\\
A^{T}{\bar e}&=&{\bar v}\,.
\end{subeqnarray}
\end{lemma}

\begin{proof}
Suppose the network~\eqref{eqn:nonstar} is not synchronous. Then, by
Remark~\ref{rem:passive} and Remark~\ref{rem:nominal}, there must
exist a solution $v(t)={\rm Re}({\bar v} e^{j\omega t})$ satisfying
\eqref{eqn:nonstar} with some $e(t)={\rm Re}({\bar e} e^{j\omega
t})$, where $\omega\in\Real$ and ${\bar v}\notin{\rm
span}(\one_{q})$. Substituting this particular pair $(v(t),\,e(t))$
into \eqref{eqn:nonstar} yields
\begin{subeqnarray}\label{eqn:substitute}
(\omega_{0}^{2}-\omega^{2})AA^{T}{\bar e}+j\omega G{\bar e}+B{\bar e}&=&0\\
{\bar v}&=&A^{T}{\bar e}\,.
\end{subeqnarray}
Without loss of generality let ${\bar v}$ be a unit vector, i.e.,
${\bar v}^{*}{\bar v}=1$. Now, multiplying (\ref{eqn:substitute}a)
from left with ${\bar e}^{*}$ and rearranging the terms allow us to
write
\begin{eqnarray}\label{eqn:see}
\omega^{2}-\omega_{0}^{2}={\bar e}^{*}B{\bar e}+j\omega{\bar
e}^{*}G{\bar e}\,.
\end{eqnarray}
Since $G$ and $B$ are symmetric positive semidefinite, we have
${\bar e}^{*}B{\bar e}\geq 0$ and ${\bar e}^{*}G{\bar e}\geq 0$.
Then \eqref{eqn:see} lets us see that $\omega\geq \omega_{0}$,
${\bar e}^{*}G{\bar e}=0$, and consequently $G{\bar e}=0$. Finally,
combining $G{\bar e}=0$ with \eqref{eqn:substitute} yields
\eqref{eqn:ebar}.

To show the other direction suppose \eqref{eqn:ebar} is satisfied by
some choice of parameters $\omega$, ${\bar e}$, and ${\bar
v}\notin{\rm span}\{\one_{q}\}$. Construct the signals $v(t)={\rm
Re}({\bar v} e^{j\omega t})$ and $e(t)={\rm Re}({\bar e} e^{j\omega
t})$, which clearly satisfy \eqref{eqn:nonstar}. Hence $v(t)={\rm
Re}({\bar v} e^{j\omega t})$ is a possible solution. Then (since the
network is LTI) by Remark~\ref{rem:nominal} the signal ${\hat
v}(t)={\rm Re}({\bar v} e^{j\omega t})+\one_{q}\sin(\omega_{0}t)$ is
also a possible solution, through which we see that the
network~\eqref{eqn:nonstar} cannot be synchronous, because ${\bar
v}\notin{\rm span}\{\one_{q}\}$.
\end{proof}

\begin{lemma}\label{lem:royce}
There exist $\omega\in\Real$, ${\bar e}\in\Complex^{n}$, and ${\bar
v}\in\Complex^{q}\setminus{\rm span}\{\one_{q}\}$ satisfying
\eqref{eqn:ebar} if and only if $Y$ has two or more eigenvalues on
the imaginary axis.
\end{lemma}

\begin{proof}
By Theorem~\ref{thm:Lambda} we have $Y\one_{q}=0$. Therefore
$\lambda_{1}=0$ is an eigenvalue of $Y$ with the eigenvector
$\one_{q}$. Suppose now this eigenvalue at the origin is not the
only eigenvalue on the imaginary axis. That is, there exists a
second eigenvalue $\lambda_{2}=j\mu$ with $\mu\in\Real$. (We note
that $\mu\geq 0$ by Theorem~\ref{thm:Lambda}.) This implies there
exists a unit eigenvector ${\bar v}\notin{\rm span}\,\{\one_{q}\}$
satisfying $Y{\bar v}=j\mu{\bar v}$. This is obvious if
$\lambda_{2}\neq 0$. To see that it is still true even if the
eigenvalue at the origin is repeated (i.e., $\lambda_{2}=0$) suppose
otherwise. That is, $\one_{q}$ is the only eigenvector for the
repeated eigenvalue at the origin. This requires that there exists a
generalized eigenvector $w$ satisfying $Y w=\one_{q}$. But then the
symmetry $Y=Y^{T}$ produces the contradiction
$0=(Y\one_{q})^{T}w=\one_{q}^{T}(Y w)=\one_{q}^{T}\one_{q}=q$.
Consider now \eqref{eqn:Lambda}, which is satisfied with some
$E\in\Complex^{n\times q}$. Let ${\bar e}=E{\bar v}$. Multiplying
both sides of \eqref{eqn:Lambda} from right by ${\bar v}$ yields
\begin{eqnarray*}
\left[\begin{array}{cc} G+jB & -A\\A^{T}& 0
\end{array}\right]\left[\begin{array}{c} {\bar e} \\ j\mu{\bar v}
\end{array}\right]=\left[\begin{array}{c} 0 \\ {\bar v}
\end{array}\right]
\end{eqnarray*}
whence we extract
\begin{subeqnarray}\label{eqn:rain}
G{\bar e}+j(B{\bar e}-\mu A{\bar v})&=&0\\
A^{T}{\bar e}&=&{\bar v}
\end{subeqnarray}
Using these identities and ${\bar v}^{*}{\bar v}=1$ we can write
\begin{eqnarray*}
0
&=&{\bar e}^{*}G{\bar e}+j{\bar e}^{*}(B{\bar e}-\mu A{\bar v})\\
&=&{\bar e}^{*}G{\bar e}+j({\bar e}^{*}B{\bar e}-\mu (A^{T}{\bar
e})^{*}{\bar v})\\
&=&{\bar e}^{*}G{\bar e}+j({\bar e}^{*}B{\bar e}-\mu)\,.
\end{eqnarray*}
Since $G$ and $B$ are symmetric positive semidefinite matrices, we
have to have ${\bar e}^{*}G{\bar e}=0$ yielding
\begin{eqnarray}\label{eqn:Gbare}
G{\bar e}=0
\end{eqnarray}
Using \eqref{eqn:Gbare}, \eqref{eqn:rain}, and letting
$\omega=\sqrt{\omega_{0}^{2}+\mu}$ we can write
\begin{eqnarray}\label{eqn:17a}
((\omega_{0}^{2}-\omega^{2})AA^{T}+B){\bar e}&=&0
\end{eqnarray}
Combining (\ref{eqn:rain}b), \eqref{eqn:Gbare}, and \eqref{eqn:17a}
then yields \eqref{eqn:ebar}.

Now we show the other direction. Suppose \eqref{eqn:ebar} holds for
some $\omega\in\Real$, ${\bar e}\in\Complex^{n}$, and ${\bar
v}\in\Complex^{q}\setminus{\rm span}\,\{\one_{q}\}$. Defining the
real number $\mu=\omega^{2}-\omega_{0}^{2}$ we can mold
\eqref{eqn:ebar} into
\begin{eqnarray*}
\left[\begin{array}{cc} G+jB & A\\-A^{T}& 0
\end{array}\right]\left[\begin{array}{c} -{\bar e} \\ j\mu{\bar v}
\end{array}\right]=\left[\begin{array}{c} 0 \\ I
\end{array}\right]{\bar v}\,.
\end{eqnarray*}
Choose some $E\in\Complex^{n\times q}$ satisfying
\eqref{eqn:Lambda}. Using the symmetries of $G,\,B,\,Y$ we can write
\begin{eqnarray*}
Y {\bar v}&=&[E^{T}\ Y]\left[\begin{array}{c} 0 \\ I
\end{array}\right]{\bar v}\\
&=&[E^{T}\ Y]\left[\begin{array}{cc} G+jB & A\\-A^{T}& 0
\end{array}\right]\left[\begin{array}{c} -{\bar e} \\ j\mu{\bar v}
\end{array}\right]\\
&=&[0\ \ I]\left[\begin{array}{c} -{\bar e} \\ j\mu{\bar v}
\end{array}\right]\\
&=&j\mu{\bar v}\,.
\end{eqnarray*}
Recall $Y\one_{q}=0$. Then $Y{\bar v}=j\mu{\bar v}$ implies $Y$ has
at least two eigenvalues on the imaginary axis because ${\bar
v}\notin{\rm span}\,\{\one_{q}\}$ and $\mu$ is real.
\end{proof}

\vspace{0.12in}

\noindent {\bf Proof of Theorem~\ref{thm:sync}.} Combine
Lemma~\ref{lem:rolls} and
Lemma~\ref{lem:royce}.\hfill\null\hfill$\blacksquare$

\section{An example}\label{sec:ex}

In this section we provide an illustration of
Theorem~\ref{thm:sync}, where an example network of four harmonic
oscillators under bilayer RL coupling is studied.

\begin{figure}[h]
\begin{center}
\includegraphics[scale=0.5]{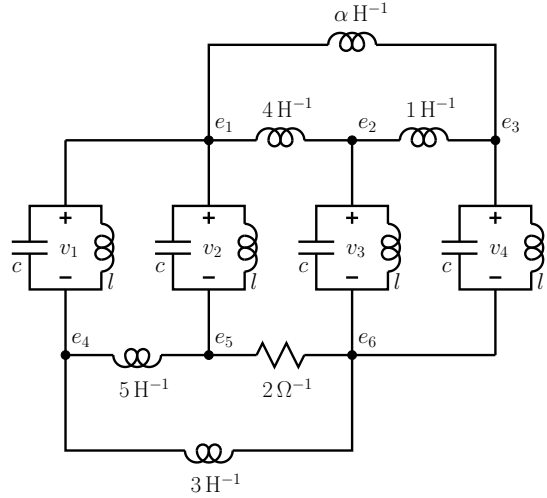}
\caption{A network of identical LC tanks under bilayer RL
coupling.}\label{fig:LCbilayer}
\end{center}
\end{figure}

Consider the network of $q=4$ coupled LC tanks shown in
Fig.~\ref{fig:LCbilayer}. The network has $n=6$ nodes. For the
labeling shown in the figure the matrices $G\in\Real^{6\times 6}$,
$B\in\Real^{6\times 6}$, and $A\in\Real^{6\times 4}$ enjoy the
structures given in \eqref{eqn:forms} where the submatrices read
\begin{eqnarray*}
\begin{array}{rclrclrcl}
G_{1}&=&\left[\begin{array}{rrr}0&0&0\\0&0&0\\0&0&0\end{array}\right]\,,\quad
&B_{1}&=&\left[\begin{array}{crc}\alpha+4&-4&-\alpha\\-4&5&-1\\-\alpha&-1&\alpha+1\end{array}\right]\,,\quad
&F_{1}&=&\left[\begin{array}{rrrr}1&1&0&0\\0&0&1&0\\0&0&0&1\end{array}\right]\\
\\
G_{2}&=&\left[\begin{array}{rrr}0&0&0\\0&2&-2\\0&-2&2\end{array}\right]\,,\quad
&B_{2}&=&\left[\begin{array}{rrr}8&-5&-3\\-5&5&0\\-3&0&3\end{array}\right]\,,\quad
&F_{2}&=&\left[\begin{array}{rrrr}1&0&0&0\\0&1&0&0\\0&0&1&1\end{array}\right]\,.
\end{array}
\end{eqnarray*}
For the parameter choice $\alpha=1$ the eigenvalues of the effective
Laplacian $Y\in\Complex^{4\times 4}$ can be computed to be
$\{\lambda_{1}=
0,\,\lambda_{2}=0.5795+j1.8886,\,\lambda_{3}=0.6283+j4.1990,\,\lambda_{4}=1.4393
+j11.3242$. Since $\lambda_{1}=0$ is the only eigenvalue on the
imaginary axis, by Theorem~\ref{thm:sync} we can say that the
oscillators will asymptotically synchronize when $\alpha=1$. For
$\alpha=4$, however, the eigenvalues read
$\lambda_{1}=0,\,\lambda_{2}=j6,\,\lambda_{3}=1.1989+j11.3818,\,\lambda_{4}=1.3931+j2.3622$.
This time there are two eigenvalues on the imaginary axis, namely,
$\lambda_{1}=0$ and $\lambda_{2}=j6$. Therefore the oscillators are
not guaranteed to synchronize for this case. This example tells us
that synchronization cannot be determined merely by the structure of
the coupling. In other words, without the actual parameter values,
knowing only which oscillator is connected to which and by what type
of connector is in general not sufficient to make definite
conclusions about the collective behavior of the oscillators.

\section{Conclusion}

In this paper we studied networks of coupled LC tanks with nonstar
oscillator graph. We first considered the special case, where the
coupling is purely resistive. For such networks we showed that the
interconnection has to be bilayer in order for the oscillator
voltages to asymptotically synchronize. Then we moved on to analyze
the general case (where both resistive coupling and inductive
coupling are simultaneously active) under bilayer coupling
structure. The layered architecture generates six real matrices
(three for each layer) which do not readily tell whether the
oscillators achieve synchronization or not. To overcome this
difficulty we proposed a method to construct a single complex matrix
(called the effective Laplacian) out of those six matrices and
presented a simple test to study synchronization. The test is this.
The oscillators synchronize if and only if the effective Laplacian
has a single eigenvalue on the imaginary axis.

\bibliographystyle{plain}
\bibliography{references}

\begin{thebibliography}{10}

\bibitem{achanta18}
P.~Achanta, M.~Sinha, B.~Johnson, S.~Dhople, and D.~Maksimovic.
\newblock Self-synchronizing series-connected inverters.
\newblock In {\em Proc. of the IEEE 19th Workshop on Control and Modeling for
  Power Electronics}, 2018.

\bibitem{asratian98}
A.S. Asratian, T.M.J. Denley, and R.~Haggkvist.
\newblock {\em Bipartite Graphs and Their Applications}.
\newblock Cambridge University Press, 1998.

\bibitem{berkics17}
P.~Berkics.
\newblock On parallel sum of matrices.
\newblock {\em Linear and Multilinear Algebra}, 65:2114--2123, 2017.

\bibitem{cejnar20}
P.~Cejnar, O.~Vysata, J.~Kukal, M.~Beranek, and M.~Valis~A. Prochazka.
\newblock Simple capacitor-switch model of excitatory and inhibitory neuron
  with all parts biologically explained allows input fire pattern dependent
  chaotic oscillations.
\newblock {\em Scientific Reports}, 10:7353, 2020.

\bibitem{dorfler14}
S.V. Dhople, B.B. Johnson, F.~D\"{o}rfler, and A.O. Hamadeh.
\newblock Synchronization of nonlinear circuits in dynamic electrical networks
  with general topologies.
\newblock {\em IEEE Transactions on Circuits and Systems I: Regular Papers},
  61:2677--2690, 2014.

\bibitem{liu20}
P.~Liu, L.~Song, and S.~Duan.
\newblock A synchronization method for the modular series-connected inverters.
\newblock {\em IEEE Transactions on Power Electronics}, 35:6686--6690, 2020.

\bibitem{narahara19}
K.~Narahara.
\newblock Dynamics of traveling pulses developed in a tunnel diode oscillator
  ring for multiphase oscillation.
\newblock {\em Nonlinear Dynamics}, 95:2729--2743, 2019.

\bibitem{peles03}
S.~Peles and K.~Wiesenfeld.
\newblock Synchronization law for a van der {P}ol array.
\newblock {\em Physical Review E}, 68:026220, 2003.

\bibitem{ren08}
W.~Ren.
\newblock Synchronization of coupled harmonic oscillators with local
  interaction.
\newblock {\em Automatica}, 44:3195--3200, 2008.

\bibitem{su09}
H.~Su, X.~Wang, and Z.~Lin.
\newblock Synchronization of coupled harmonic oscillators in a dynamic
  proximity network.
\newblock {\em Automatica}, 45:2286--2291, 2009.

\bibitem{tuna17}
S.E. Tuna.
\newblock Synchronization of harmonic oscillators under restorative coupling
  with applications in electrical networks.
\newblock {\em Automatica}, 75:236--243, 2017.

\bibitem{vandersteen06}
R.~v.d. Steen and H.~Nijmeijer.
\newblock Partial synchronization of diffusively coupled {C}hua systems: an
  experimental case study.
\newblock {\em IFAC Proceedings}, 39:119--124, 2006.

\bibitem{wu01}
C.W. Wu.
\newblock Synchronization in arrays of coupled nonlinear systems: passivity,
  circle criterion, and observer design.
\newblock {\em IEEE Transactions on Circuits and Systems I: Fundamental Theory
  and Applications}, 48:1257--1261, 2001.

\bibitem{zhou12}
J.~Zhou, H.~Zhang, L.~Xiang, and Q.~Wu.
\newblock Synchronization of coupled harmonic oscillators with local
  instantaneous interaction.
\newblock {\em Automatica}, 48:1715--1721, 2012.

\end{thebibliography}
\end{document}